\documentclass[smallextended,referee,envcountsect]{svjour3}
\usepackage{paralist} 
\usepackage{amssymb}
\usepackage{authblk}
\usepackage{multirow,booktabs}
\usepackage{graphicx}
\graphicspath{{figures/}}
\usepackage{subfigure}
\usepackage{epsfig}
\usepackage{epstopdf}

\usepackage{float, enumerate, graphicx, amssymb, amsmath , color, longtable, multirow, lineno,url}
\usepackage{graphicx}
\usepackage[justification=centering]{caption}
\usepackage{amsmath}

\usepackage[active]{srcltx}
\usepackage{pifont}
\usepackage{longtable,booktabs}

\numberwithin{assumption}{section}
\numberwithin{equation}{section}
\usepackage{algorithm}
\usepackage{algorithmic}
\usepackage[a4paper,left=2cm,right=2cm,bottom=2.3cm,top=2.3cm]{geometry}
\usepackage[colorlinks,linkcolor=blue,citecolor=blue,                  bookmarksnumbered=true,
bookmarksopen=true,
colorlinks ]{hyperref}

\usepackage[a4paper,left=2cm,right=2cm,bottom=2.3cm,top=2.3cm]{geometry}

\begin{document}
	\title{ On the properties    of the linear conjugate gradient method   }
	\author{Zexian Liu \and Qiao Li    }
	\institute{
		Zexian Liu,   e-mail: liuzexian2008@163.com(\ding{41});  	Qiao Li, e-mail:liqiaomuzi@163.com
		\at  School of Mathematics and Statistics, Guizhou University, Guiyang, 550025, China
	}
	\date{Received: date / Accepted: date}
	
	\maketitle
	
	\begin{abstract}   
	The linear conjugate gradient method is an efficient iterative     method  for  	
	  the   convex quadratic minimization problems $ \mathop {\min }\limits_{x \in { \mathbb R^n}} f(x) =\dfrac{1}{2}x^TAx+b^Tx $,   where $ A \in R^{n \times  n} $ is symmetric  and positive definite and $ b \in R^n $. It is generally   agreed  that  the gradients $ g_k $  are not conjugate with respective to $ A $  in  the   linear conjugate gradient method (see page 111 in Numerical optimization (2nd, Springer, 2006) by  Nocedal and  Wright). In the paper 	
		we prove   the conjugacy  of the gradients $ g_k   $   generated by the linear conjugate gradient method, namely,  $$g_k^TAg_i=0, \;  i=0,1,\cdots, k-2.$$ In addition,a new way is exploited to derive the linear conjugate gradient method  based on  the conjugacy of the search directions and the orthogonality of the gradients, rather than the conjugacy of the search directions and the exact stepsize. 
		
	\end{abstract}
	\section{Introduction}
	\vspace {-0.2cm}
	
	We consider the following   convex quadratic optimization problem: 
	\begin{equation} \label{eq:Quadpro}
		\mathop {\min }\limits_{x \in {  R^n}} f(x) =\dfrac{1}{2}x^TAx+b^Tx,
	\end{equation}
	\noindent where $ A \in R^{n \times  n} $ is symmetric  and positive definite and $ b \in R^n $.

The linear conjugate gradient method for solving \eqref{eq:Quadpro} has the following form 
		\begin{equation}\label{eq:iterform}
			\begin{array}{l}
				{{x_{k+1}} = {x_k} +  \alpha{_k}{d_k}},~~~~~~~~~~~ 
			\end{array}   
		\end{equation}
		where 	$ d_k  $ is the search direction given by	
		\begin{equation}\label{eq:dk} {d_k} = \left\{ {\begin{array}{*{20}{c}}
			{ - {g_0},\;\;\;\;\;\;\;\;\;\;\;\;\;\;\;\;\;\;\;{\rm{if}}\;k = 0,}\\
			{ - {g_k} + {\beta _k}{d_{k - 1}},\;\;\;{\rm{if}}\;k > 0,}
			\end{array}} \right.\end{equation}
	and	$ \alpha_k  $ is the stepsize determined by the exact line search, namely, 
	\begin{equation} \label{eq:alpha}
	 {\alpha _k} =-\dfrac{g_k^Td_k}{d_k^TAd_k}= \arg \mathop {\min }\limits_{\alpha  > 0} f\left( {{x_k} + \alpha {d_k}} \right). \end{equation}
	It follows from \eqref{eq:alpha} that 	\begin{equation} \label{eq:exact}     g_{k+1}^Td_k =0.  
	\end{equation}
	   Here the parameter $ \beta_k $ is often derived by making it      satisfy   the following condition: 
	\begin{equation} \label{eq:conjugacy}  d_{k+1}^TAd_k =0,  \end{equation}
	which together with \eqref{eq:exact} can yield  the conjugacy of the whole search directions, as well as the global convergence. 
	Some  well-known formulae for $ \beta_k $ are called the
 Fletcher-Reeves (FR) \cite{Fletcher1964Function}, Hestenes-Stiefel (HS)   \cite{Hestenes1952Methods}, Polak-Ribi\`ere-Polyak (PRP)   \cite{Polak1969The,Polak1969Note}   and Dai-Yuan (DY)  \cite{Dai1999A}
 formulae, and are given by
 \[\beta _k^{FR} = \frac{{{{\left\| {{g_{k}}} \right\|}^2}}}{{{{\left\| {{g_{k-1  }}} \right\|}^2}}},\;\;\;\;\beta _k^{HS} = \frac{{g_{k }^T{y_{k-1  }}}}{{d_{k-1  }^T{y_{k-1  }}}},\;\;\;\beta _k^{PRP} = \frac{{g_{k}^T{y_{k-1 }}}}{{{{\left\| {{g_{k-1 }}} \right\|}^2}}},\;\;\;\beta _k^{DY} = \frac{{{{\left\| {{g_{k}}} \right\|}^2}}}{{d_{k-1  }^T{y_{k-1  }}}},\]
 where $y_{k-1} =g_k -g_{k-1}$, and  $\left\| \cdot \right\| $ denotes the    Euclidean norm.  In the case that $ f $ is given by \eqref{eq:Quadpro} and the exact line search \eqref{eq:alpha} is performed, all $ \beta_k $ should be the same.

	It is well-known that the linear conjugate gradient method enjoys the following nice properties.  
	  \begin{theorem}\label{thm 2-1}    \textbf{  \cite{DaiYuanCG,Yuan1999Theory}}   Suppose that the  iterate $ \left\lbrace x_k \right\rbrace  $ is generated by the linear  conjugate gradient method   for solving \eqref{eq:Quadpro}, and $ x_k $ is not the solution  point $ x^* $. Then,
\[\begin{array}{l}
g_i^T{d_i} =  - {\left\| {{g_i}} \right\|^2},\;\;i = 0,1,2, \cdots ,\\
d_i^TA{d_j} = 0,\;\;j = 0,1,2, \cdots ,i - 1,\\
g_i^T{d_j} = 0,\;\;j = 0,1,2, \cdots ,i - 1,\;\;\\
g_i^T{g_j} = 0,\;\;j = 0,1,2, \cdots ,i - 1.
\end{array}\]
	 Further, the sequence $  \left\lbrace  x_k \right\rbrace   $ converges to $ x^*$ in at most $ n $ steps. 
	  \end{theorem}
	  
	  It follows from Theorem \ref{thm 2-1} that the gradients $ g_k $    are mutually orthogonal in the linear conjugate gradient method. However, it is generally   agreed  that  the gradients $ g_k $ generated by  the   linear conjugate gradient method are not conjugate with respective to $ A $. For example, one can find the well-known description	``\emph{Since the gradients $ r_k $ are mutually orthogonal, 
	 the term ``conjugate gradient method'' is actually a \textbf{misnomer}. It is the search directions, not the gradients, that are conjugate with	respect to A.}\ ''  in Page 111 of   Nocedal and Wright's  monograph
	\cite{Nocedal2008}.   Must the gradients generated by the linear conjugate gradient method  not be    conjugate with respective to the matrix $A$ ? In the paper we will  give a 
	negative answer to this question, as well as exploiting a new way for deriving the linear conjugate gradient method  based on  the conjugacy of the search direction  and the orthogonality of the gradient, rather than the conjugacy of the search direction and the exact stepsize.

	\section{The conjugacy of the  gradients }
	In the section  we will establish the conjugacy of the gradients generated by the linear conjugate gradient method.
	\begin{theorem} \label{thm: orthogonalCG}
		Suppose that  $ \left\lbrace x_k\right\rbrace \;  \text{and}\; \left\lbrace g_k\right\rbrace   $     are generated by linear conjugate gradient method  for   solving \eqref{eq:Quadpro}.    Then,  	 
		\begin{equation}  \label{eq:gggdgAg=0CG}
			\begin{array}{l}
				g_{k + 1}^TA{g_k} =  - \dfrac{{\left\| {{g_{k + 1}}} \right\|_2^2}}{{{\alpha _{k}}}}\  ,\;\\
				g_{k+1}^TA{g_i} = 0,\;i=0, 1,  \cdots,    k-1.
			\end{array}   
		\end{equation}
		
	\end{theorem}

	\begin{proof} We first prove the first equality in \eqref{eq:gggdgAg=0CG}.   
	If $ k=0 $, then it follows  from \eqref{eq:dk} and Theorem \ref{thm 2-1} that 
	\[g_1^TA{g_0} =  - g_1^TA{d_0} =  - \frac{1}{{{\alpha _0}}}g_1^T\left( {{g_1} - {g_0}} \right) =  - \frac{{\left\| g_1 \right\|^2}}{{{\alpha _0}}}.\]
	When $ k>0 $, we also obtain from  the orthogonality of gradients and \eqref{eq:dk} that 	  
				\begin{align*} 
				g_{k + 1}^TA{g_k} & = g_{k + 1}^TA( - {d_k} + \beta{_k}{d_{k - 1}})\\
				& =  - \frac{{g_{k + 1}^T({g_{k + 1}} - {g_k})}}{{\alpha{_k}}} + \frac{{\beta{_k}g_{k + 1}^T({g_k} - {g_{k - 1}})}}{{\alpha{_{k - 1}}}}\\
				& =  - \frac{{\left\| {{g_{k + 1}}} \right\|_2^2}}{{{\alpha _k}}}.
				\end{align*} 
	We then prove the second equality in \eqref{eq:gggdgAg=0CG}. 	
		 If $ k=1 $,   by \eqref{eq:dk} and  Theorem \ref{thm 2-1}, we obtain 	\begin{equation} \label{eq:aAd}
		g_2^TA{g_0} =  - g_2^TA{d_0} =  - \frac{1}{{{\alpha _0}}}g_2^T\left( {{g_1} - {g_0}} \right) = 0.
		\end{equation}	
When $ k>1 $, it follows from   \eqref{eq:dk}  and  Theorem \ref{thm 2-1}  that  
				\begin{align*} 
				g_{k + 1}^TA{g_i} & = g_{k + 1}^TA( - {d_i} + \beta{_i}{d_{i - 1}})\\
				& =  - \frac{{g_{k + 1}^T({g_{i + 1}} - {g_i})}}{{\alpha{_i}}} + \frac{{\beta{_i}g_{k + 1}^T({g_i} - {g_{i - 1}})}}{{\alpha{_{i - 1}}}}\\
				& = 0,   
				\end{align*} 
where $ \;i=0, 1,  \cdots,    k-1. $  It completes the proof.	$\hfill{} \Box$

	\end{proof}
	
\noindent \textbf{Remark 1.}	It  follows from Theorem  \ref{thm: orthogonalCG} that $g_{k+1} $ is  conjugate to $ g_{i} \;( i=0,1,\cdots,k-1)$ with respective to $A$. Due to the conjugacy of the gradients, it seems that  it is also reasonable to name the iterative method \eqref{eq:iterform} and  \eqref{eq:dk} with \eqref{eq:alpha}    the  linear conjugate gradient method. 
	
		\section{ New way for deriving the linear   conjugate gradient method    }
	
	In the section, we will exploit  a new way for deriving the linear   conjugate gradient method.  
	
	The orthogonality of gradients is crucial to the linear conjugate gradient method. Different from the stepsize \eqref{eq:alpha} which makes $ f $ reduce mostly along the search direction $ d_k $,   the new  stepsize $ \alpha_k $ in  \eqref{eq:dk}  is choose such that the resulting new gradient  $ g_{k+1} $ is orthogonal to the latest gradient $ g_k $, namely, 
 	\begin{equation} \label{eq:orhogg} g_{k+1}^Tg_k =0,\end{equation}	  
which  implies that
	\begin{equation} \label{eq:newak}	  \alpha_k= -\dfrac{g_k^Tg_k }{g_k^TAd_k}. 	\end{equation}

\noindent Therefore,  	the new iterative method has the form  \eqref{eq:iterform} and  \eqref{eq:dk} with \eqref{eq:newak}, which    is derived based on  the   conditions:
	\begin{equation} \label{eq:condition2}  g_{k+1}^Tg_k =0 \;\; \text{and} \;d_{k+1}^TAd_k =0. \;   \end{equation}
	Note that it is   different from the conditions  \eqref{eq:exact}  and \eqref{eq:conjugacy}  used to derive the   linear conjugate gradient method.
	
	In the following theorem, we will prove the equivalence of the iterative method \eqref{eq:iterform} and  \eqref{eq:dk} with \eqref{eq:newak} and the linear conjugate gradient method.

\begin{theorem} \label{thm: equvalence}
 The new stepsize \eqref{eq:newak} is also the exact stepsize, namely, 	 $ {\alpha _k}  = \arg \mathop {\min }\limits_{\alpha  > 0} f\left( {{x_k} + \alpha {d_k}} \right). $

\end{theorem} 
 
 	\begin{proof}
 		It 	follows from $d_k^TAd_{k-1} =0  $ that   	$  	\alpha_k= \dfrac{g_k^Tg_k }{d_k^TAd_k}  $.  Thus, we have 		
\[\begin{array}{l}
g_{k + 1}^T{d_k} = {\left( {{g_k} + {\alpha _k}A{d_k}} \right)^T}{d_k}\\
\;\;\;\;\;\;\;\;\; = g_k^T{d_k} + \frac{{g_k^T{g_k}}}{{d_k^TA{d_k}}}d_k^TA{d_k}\\
\;\;\;\;\;\;\;\;\; = {\beta _k}g_k^T{d_{k - 1}}\\
\;\;\;\;\;\;\;\;\; = {\beta _k}{\beta _{k - 1}} \cdots {\beta _1}g_1^T{d_0}\\
\;\;\;\;\;\;\;\;\; =  - {\beta _k}{\beta _{k - 1}} \cdots {\beta _1}g_1^T{g_0}\\
\;\;\;\;\;\;\;\;\; = 0,
\end{array}\]	
which togetcher with \eqref{eq:Quadpro} yields that $$ {\alpha _k}  = \arg \mathop {\min }\limits_{\alpha  > 0} f\left( {{x_k} + \alpha {d_k}} \right).   $$
It completes the proof.	$\hfill{} \Box$
	\end{proof}	
	
\noindent \textbf{Remark 2.}	It follows from Theorem \ref{thm: equvalence} that the iterative method \eqref{eq:iterform} and  \eqref{eq:dk} with \eqref{eq:newak} is indeed the linear conjugate gradient method.   Thus we can easily obtain the following corollary.	

\noindent \textbf{Remark 3.} The above new way for deriving the  linear conjugate gradient method might motivate  us to    design a new criterion that the new trial gradient  satisfies,     which   might    make   the resulting   conjugate gradient method enjoy finite temination property without the exact line search, as well as nice numerical performance.
	
		\begin{corollary} 
			  Suppose that the  iterate $ \left\lbrace x_k \right\rbrace  $ is generated by  the iterative method \eqref{eq:iterform} and  \eqref{eq:dk} with \eqref{eq:newak} for solving  \eqref{eq:Quadpro}, and $ x_k $ is not the solution  point $ x^* $. Then,
			  the sequence $  \left\lbrace  x_k \right\rbrace   $ converges to $ x^*$ in at most $ n $ steps. 
			
	\end{corollary}

		\section{Conclusion and Discussion}
In the paper we establish the conjugacy of the gradients generated by the linear conjugate gradient method, and give a new way for deriving the linear conjugate gradient method based on the orthogonality of the gradients and the conjugacy of the search dirction at the most recent iterations.

However, there are still many questions under investigation. For example, how to design a new  design   that the new trial gradient  satisfies,     which   might    make   the resulting   conjugate gradient method enjoy nice theoretical property, as well as good numerical performance.

		\begin{acknowledgements}    This research is supported by National Science Foundation of China (No. 12261019), Guizhou Provincial Science and Technology Projects (No. QHKJC-ZK[2022]YB084).
		\end{acknowledgements}
		
		\noindent\textbf{Data availability.}					
		The datasets generated during and/or analysed during the current study are available from the corresponding author on reasonable request.	
		
		\noindent\textbf{Conflict of interest.	}
		The authors declare no competing interests.

	\end{document}